\documentclass{amsart}
\usepackage{amsmath}
\usepackage{amssymb}
\usepackage[all]{xy}
\usepackage{graphicx}
\usepackage{mathrsfs}

\numberwithin{equation}{section}

 \usepackage{url}

\usepackage[latin1]{inputenc}
\usepackage{xspace,amssymb,amsfonts,euscript}
\usepackage{amsthm,amsmath}
\usepackage{palatino}
\usepackage{euscript}
\input xy \xyoption {all}

\usepackage{tikz,environ}
\usetikzlibrary{patterns,snakes}

\RequirePackage{color}
\definecolor{myred}{rgb}{0.75,0,0}
\definecolor{mygreen}{rgb}{0,0.5,0}
\definecolor{myblue}{rgb}{0,0,0.65}

\RequirePackage{ifpdf}
\ifpdf
 \IfFileExists{pdfsync.sty}{\RequirePackage{pdfsync}}{}
 \RequirePackage[pdftex,
  colorlinks = true,
  urlcolor = myblue, 
  citecolor = mygreen, 
  linkcolor = myred, 
  pagebackref,
  pdfpagemode=None,
  bookmarksopen=true]{hyperref}
\else
 \RequirePackage[hypertex]{hyperref}
\fi

\RequirePackage{ae, aecompl, aeguill} 

    \def\CM{{\mathbb{C}}}

    \def\PM{{\mathbb{P}}}
    \def\QM{{\mathbb{Q}}}

    \def\ZM{{\mathbb{Z}}}

\def\a{\alpha}

\def\g{\gamma}

\def\e{\varepsilon}

\def\la{\lambda}

\def\z{\zeta}

\newcommand{\nc}{\newcommand} \newcommand{\renc}{\renewcommand}

\newcommand{\rdots}{\mathinner{ \mkern1mu\raise1pt\hbox{.}
    \mkern2mu\raise4pt\hbox{.}
    \mkern2mu\raise7pt\vbox{\kern7pt\hbox{.}}\mkern1mu}}

\def\un{\underline}

\def\to{\rightarrow}

\def\laongto{\laongrightarrow}

\nc{\triright}{\stackrel{[1]}{\to}}
\nc{\laongtriright}{\stackrel{[1]}{\laongto}}

\nc{\Hb}{H^\bullet}

\nc{\Br}{\mathcal{B}}
\nc{\HotRR}{{}_R\mathcal{K}_R}
\nc{\HotR}{\mathcal{K}_R}
\nc{\excise}[1]{}
\nc{\defect}{\text{df}}
\nc{\h}[1]{\underline{H}_{#1}}

\nc{\Ga}{\mathbb{G}_a} 
\nc{\Gm}{\mathbb{G}_m} 

\nc{\Perv}{{\mathbf{P}}}

\nc{\IH}{{\mathrm{IH}}}

\nc{\ic}{\mathbf{IC}}

\nc{\gl}{{\mathfrak{gl}}}
\renc{\sl}{{\mathfrak{sl}}}
\renc{\sp}{{\mathfrak{sp}}}

\renc{\Im}{\textrm{Im}}

\nc{\HBM}{H^{BM}}

\DeclareMathOperator{\codim}{{\mathrm{codim}}}

\newtheorem{thm}{Theorem}[section]
\newtheorem{lem}[thm]{Lemma}

\newtheorem{prop}[thm]{Proposition}
\newtheorem{cor}[thm]{Corollary}

\theoremstyle{definition}

\newtheorem{ex}[thm]{Example}

\theoremstyle{remark}
\newtheorem{remark}[thm]{Remark}
\newtheorem{question}[thm]{Question}

\newcommand{\into}{\hookrightarrow}

\def\pt{{\mathrm{pt}}}

\nc{\simto}{\stackrel{\sim}{\to}}

 \title{On torsion in the intersection cohomology of Schubert varieties}

\author{Geordie Williamson}
\email{geordie@mpim-bonn.mpg.de}
\address{Max-Planck-Institut f\"ur Mathematik, Vivatsgasse 7, 53111,
  Bonn, Germany.}

\begin{document}

\begin{abstract} We prove that the prime torsion in the
  local integral intersection cohomology of Schubert varieties in the flag variety of
  the general linear group grows exponentially in the rank. The idea
  of the proof is to find a highly singular point in a Schubert
  variety and calculate the Euler class of the normal bundle to the
  (miraculously smooth) fibre in a particular
   Bott-Samelson resolution. The result is a geometric 
  version of an earlier result established using Soergel bimodule techniques.
\end{abstract}

\maketitle

\begin{center}
  \emph{Dedicated to the memory of Sandy Green.}
\end{center}

\section{Introduction}

Let $X$ be a projective complex algebraic variety equipped with its
metric topology and let $H^*(X, \QM)$ denote its rational
cohomology ring. If $X$ is smooth then there are several remarkable
and useful theorems concerning $H^*(X,\QM)$: Poincar\'e duality, the
hard Lefschetz theorem, the Hodge decomposition, the Hodge-Riemann
relations. If instead one takes integral coefficients
then (derived) Poincar\'e duality still holds.

None of these theorems are valid for singular $X$. Instead one
can consider the intersection cohomology $IH^*(X,\QM)$ of $X$ as defined by
Goresky and MacPherson. If $X$ is smooth then one has a
canonical identification between cohomology and intersection
cohomology. Goresky and
MacPherson proved that Poincar\'e
duality always holds in rational intersection cohomology. It was subsequently
discovered that (analogues of the) the hard
Lefschetz theorem, the Hodge decomposition and
the Hodge-Riemann relations all hold in intersection cohomology
\cite{BBD, Saito}. Thus it is not surprising that intersection cohomology provides a
powerful complement to ordinary cohomology in the study of singular
algebraic varieties.

As with ordinary cohomology, it is also possible to define 
intersection cohomology groups with coefficients in any field or the integers. In their original paper
on intersection (co)homology, Goresky and MacPherson noticed that
(derived) Poincar\'e duality does not hold over the integers
for intersection cohomology \cite[6.3]{GM}. For example, if $X$ is smooth then Poincar\'e
duality implies that the intersection form on (the free part of) its middle cohomology is
unimodular (i.e. non-degenerate over $\ZM$). Goresky and MacPherson
gave an example to show that the analoguous statement need not hold
for integral intersection cohomology. For a given singular $X$ it is
appears to be a difficult question to decide whether its integral
intersection cohomology satisfies Poincar\'e duality
over the integers, or for which primes $p$ it fails.

This question has a local variant. Just as the ordinary cohomology
of a space can be described as the cohomology of the constant sheaf,
intersection cohomology may be obtained as the hypercohomology of the
intersection cohomology complex, a constructible complex of abelian
groups on $X$. Let $\ic(X,\ZM)$ denote the integral intersection
cohomology complex of $X$. The local variant of the above question
which we consider in this paper is the following:

\begin{question} \label{q}
  Descibe the $p$-torsion in the
  stalks or costalks of $\ic(X,\ZM)$. In particular, for which
  primes $p$ are all the stalks and costalks free of $p$-torsion?
\end{question}

Some remarks about this question are in order:
\begin{enumerate}
\item If there is no torsion in the stalks or costalks of $\ic(X,\ZM)$
  then the integral intersection cohomology  satisfies
  Poincar\'e duality. The $p$-local version of this statement also
  holds: absence of $p$-torsion implies that Poincar\'e duality holds
  after inverting all primes $\ne p$. These statements are not if and
  only if in general, however in the case of Schubert varieties (considered
  below) they are.
\item In general the rational intersection
  cohomology complex $\ic(X,\QM)$ is much easier to describe. This
  is due to the decomposition theorem \cite{BBD}, which allows one to
  compute $\ic(X,\QM)$ via resolutions. In particular in many cases
  $\ic(X,\QM)$  can be considered ``known'' and the above question asks
  whether $\ic(X,\ZM)$ contains any surprises.
\item For general $X$ this question appears to be very hard. For
  example, if $X$ has only isolated singularities then the question is
  equivalent to understanding the torsion in the cohomology of the
  links\footnote{The \emph{link} of a point $x \in X$ is defined as
    follows: first we embed an affine neighbourhood of $x$ into
    $\CM^N$ so that $x \mapsto 0$; then the link is defined to be $X
    \cap S^{2N-1}_\e$ for small $\e$, where $S^{2N-1}_\e \subset \CM^{N}$ denotes the sphere of radius
  $\e$ centred at the origin.}  to all singular
  points. In general the link of an isolated singularity can be a rather
  complicated manifold, and describing the torsion in its cohomology
  can be a difficult task.
\end{enumerate}

As well as their intrinsic interest, these questions have applications
in the modular representation theory of finite and algebraic groups. Starting with
the Kazhdan-Lusztig conjectures, intersection cohomology methods have
been very fruitful in Lie theory (see \cite{Licm} for an impressive list of
applications). The power of these methods in characteristic zero
representation theory is usually thanks to the decomposition theorem.
More recently it has been suggested that similar methods could be used to attack questions in
modular representation theory \cite{Soe,MV,JuMSC,JMW1}. However here the decomposition theorem
is missing. The above questions asks for obstructions to transporting
out knowledge in characteristic zero to knowledge in characteristic $p$.

In geometric representation theory a central role is played by Schubert
varieties. The goal of this paper is to provide the following partial
answer to the above question in this case:

\begin{thm} \label{thm:intro}
  The $p$-torsion (for $p$ a prime) in the stalks and costalks of the integral intersection
  cohomology complexes on Schubert varieties in the flag variety of
  the general linear group grows at least exponentially in the rank.
\end{thm}

Again, some remarks are in order:
\begin{enumerate}
\item This is a geometric version of an earlier theorem proved using Soergel
  bimodules, diagrammatics and the nil Hecke ring in
  \cite{XW,WT}. Important contributions to this theorem were made by
  Soergel, Libedinsky, Elias-Khovanov, Elias and He. On the geometric
  side important contributions were made by
Braden (who discovered 2-torsion for $n = 8$, see the appendix to \cite{W}) and Polo (who showed the
existence of $n$-torsion in rank $4n$).
\item Schubert varieties admit affine pavings and so their
ordinary cohomology is free over the integers. The above theorem tells
us that (at least locally) there is lots of torsion in intersection
cohomology. 
\item We are still very far from a complete understanding of Question
  \ref{q} in the setting of Schubert varieties. This is already evident in the phrase ``at least'' in the
  above theorem, whose proof produces many examples of
  torsion, but certainly makes no claim to exhaustiveness.
\item In many ``simple'' examples (Schubert varieties in
  Grassmannians \cite{Zel}, Schubert varieties for $GL(n,\CM)$ for $n \le 7$ \cite{W})
  there is \emph{no torsion at all} in the stalks or costalks of integral
  intersection cohomology sheaves. The above theorem tells us that these
  simple examples are rather deceptive.
\item By results of Soergel one can use the above theorem to deduce
  that any bound for Lusztig's conjecture on the characters of simple
  rational representations of $GL_n$ in characteristic $p$ must grow at least exponentially
  in $n$. Hence the above theorem gives many counterexamples to the
  expected bounds in Lusztig's conjecture \cite{L}.
Similarly, in \cite{WT} it is explained how one can use such
  results to produce counterexamples to a conjecture of James
  \cite{James} on the simple modular representations of the symmetric group.
\item All Schubert varieties in the flag variety of $GL_n$ also occur
  as Schubert varieties in the flag varieties of groups of types
  $B_n$, $C_n$ and $D_n$. Hence the above theorem may be rephrased as
  saying that the torsion in the stalks and costalks of the integral
  intersection cohomology of Schubert varieties in the flag variety of
  any simple complex algebraic group grows at least exponentially in the rank.
\end{enumerate}

As already mentioned, the above theorem can be deduced from previous
work in the context of Soergel bimodules. However I
think it is worthwhile to publish a new proof for two reasons:
\begin{enumerate}
\item The proof relies only on the combinatorics of expressions and
  geometric ideas. In particular it does not use
  the theory of Soergel bimodules, their diagrammatics, or the nil
  Hecke ring (as in \cite{XW,WT}). Hence this paper is potentially
  accessible to a wider audience than \cite{WT}.
\item The proof provides a recipe to find many highly singular points
  in Schubert varieties, whose resolutions are nonetheless
  amenable to explicit analysis. It is possible that these points and
  their resolutions will be useful in other problems in singularity
  theory and the study of Schubert varieties. With such potential
  future applications in mind, and also with the goal of
  understanding existing work more conceptually, an explicit
  description of the geometry of the situation seems worthwhile.
\end{enumerate}

\subsection{Main theorem} We now give a more precise formulation of
the main theorem of this paper. The formulation is 
somewhat technical, and hence we need some more notation.

Let $R = \ZM[\e_1, \e_2, \dots, \e_n]$ be a polynomial ring in $n$
variables graded such that $\deg \e_i =2$ and let
$W = S_n$ the symmetric group on $n$-letters. Then $W$ acts by
permutation of variables on $R$. Let $s_1,
\dots, s_{n-1}$ denote the simple transpositions of $S_n$ and let
$\ell$ denote the  length function. Let $\partial_i$ denote the $i^{th}$ divided
difference operator:
\[
\partial_i(f) = \frac{f - s_i f}{\e_i-\e_{i+1}} \in R.
\]
For any element $w \in S_n$ we have well-defined operators $\partial_w = \partial_{i_1}
\dots \partial_{i_m}$ where $w = s_{i_1} \dots s_{i_m}$ is a reduced
expression for $w$.

Consider elements of the form
\begin{equation} \label{eq:exp}
C = \partial_{w_m} (\e_n^{a_m}  \partial_{w_{m-1}}
(\e_n^{a_{m-1}} 
\dots \partial_{w_1} (\e_n^{a_1}) \dots ))
\end{equation}
where $w_i \in S_n$ are arbitrary. 
We assume that $\sum \ell(w_i) = a$ where $a = \sum a_i$. Because
$\e_i$ has degree 2 and $\partial_{w}$ has degree $-2\ell(w)$ it
follows that $C \in \ZM$ for degree reasons.

Let $N := n + a$, and $G = GL_N(\CM)$. We identify $S_N$ (the
symmetric group on $N$ letters) with the Weyl group of $G$ in the
standard way. Let $B \subset G$ denote the Borel subgroup of
upper-triangular matrices. Given any subset $I \subset \{ 1,
\dots, N-1 \}$ we let $w_I$ denote the longest element of the standard
parabolic subgroup of $S_N$ corresponding to $I$, and by $P_I \supset
B$ denote the parabolic subgroup corresponding to $I$. Let $M := \{ 1,
\dots, n-1\}$.

The main result of this paper is the following:

\begin{thm} \label{thm:main}
  Suppose that $C \ne 0$. Then there exists a Schubert variety
  $X \subset GL_{n+a}/P_M$ and a (Bott-Samelson) resolution
\[
f : \widetilde{X} \to X \subset G/P_M
\]
such that the complex $Rf_* (\ZM/C\ZM_{\widetilde{X}})$ (the
derived pushforward of the constant sheaf on $\widetilde{X}$) is not isomorphic to a direct sum of intersection
cohomology sheaves. More precisely, the decomposition theorem fails at the
point $w_I$, where $I = \{ 1,2, \dots, n-1, n+1, \dots, n+a-1 \}$.
\end{thm}
Remarks:
\begin{enumerate}
\item The Schubert variety $X$ and resolution $\widetilde{X}$ are
explicit starting from the expression \eqref{eq:exp}
for $C$. We refer the reader \S~\ref{sec:geometry} for the
description of all spaces involved.
\item The failure of the decomposition theorem in Theorem
  \ref{thm:main} implies the existence of some intersection cohomology
  complex supported on $X$ which has $p$-torsion in its stalk or
  costalk, for some prime $p$ dividing $C$. This fact can be easily
  deduced from the theory of parity sheaves \cite{JMW2}.
\item Given the above theorem it is easy to produce many examples of
  $C$ with large prime factors relative to $N = n + a$ (see \cite[\S 6]{WT}). For example, one
  can find expressions which produce Fibonacci numbers linearly in
  $N$. However to
  \emph{prove} that these factors grow exponentially with respect to
  $N$ requires some rather sophisticated results from number theory. This is
  discussed in detail in the appendix to \cite{WT} by Kontorovich,
  McNamara and the author.
\item A technical point: The allowed expressions for $C$ in
  \cite{WT} are slightly more general than those allowed above; 
  in \cite{WT} one is also allowed expressions involving both $\e_n$
  and $\e_1$, and not only $\e_n$ as above.
  However I have checked that there are no essential gains by
  allowing these more general expressions, and the setting of
  the current paper simplifies proofs.
\end{enumerate}

\subsection{Acknowledgements:} Expressions of the
form \eqref{eq:exp} emerged first in joint work with Xuhua He
\cite{XW}. Subsequently I tried to find a geometric explanation, 
which is the Euler class lemma of \S~\ref{sec:euler-class-lemma}. I
would like to thank him for many useful discussions and
observations. I am also grateful to Tom Braden, Daniel Juteau, Carl Mautner and
Patrick Polo from whom I learnt most of the geometric and topological
techniques of this paper.

It is a great pleasure to dedicate this paper to the memory of Sandy
Green. One of my first memories of representation theory is Gus
Lehrer's empassioned description of 
Green functions and the character table of the finite general
linear group. I spent 2008 - 2011 as a postdoc in Oxford and I
remember Sandy's 
active participation in the representation theory seminar. After one
of my first lectures on parity sheaves he excitedly asked many
questions, and expressed his desire to better understand perverse
sheaves. I was impressed at his openness to new ideas, and have tried to
imitate it since.

\section{Notation}\label{sec:not}

\emph{Varieties and sheaves:}
Throughout all algebraic varieties are over $\CM$ and are equipped
with their classical (metric) topologies. Dimension and codimension
always refer to \emph{complex} dimension. Given a ring $\Lambda$
and a space $X$ we denote by $\Lambda_X$ the constant sheaf on $X$
with values in $\Lambda$.

\emph{Expressions and subexpressions:}
Throughout we view the symmetric group $S_n = W$ as a Coxeter group with
simple reflections $S \subset W$ the simple transpositions. An
expression is a sequences $\un{w} = (s_1, \dots, s_m)$ with $s_i \in
S$. We write expressions as $\un{w} = s_1 \dots s_m$ and dropping the
underline denotes the product $w \in W$. Given a fixed subexpression
$\un{w} = s_1 \dots s_m$ a subexpression is a sequence $\un{e} = e_1
\dots e_m$ with each $e_i \in \{ 0 , 1 \}$. What is traditionally referred to
as a subexpression is the sequence $(s_1^{e_1}, \dots, s_m^{e_m})$,
however we prefer the more economical notation. We write $\un{e}
\subset \un{w}$ to indicate that $\un{e}$ is a subexpression of
$\un{w}$. Given $\un{e} = e_1 \dots e_m \subset \un{w}$ we set $\un{w}^{\un{e}} =
s_1^{e_1} \dots s_m^{e_m}$.

\section{What we need to do} \label{sec:what}

In this section we recall some standard material on the role played by
intersection forms in the decomposition theorem. This section gives
the
algebro-geometric scaffolding of the rest of the paper. One can find
background material for this section in \cite{BBD, dCM,CG,JMW2}.

Fix a (singular) normal and irreducible complex algebraic variety $X$ and a resolution of
singularities\footnote{In this paper \emph{resolution of
    singularities} is used to refer to any proper birational
  morphism of algebraic varieties with smooth source. We do not
  require our map to be an isomorphism over the
  smooth locus of $X$.}
\[
f : \widetilde{X} \to X.
\]
We fix a stratification of $X$ adapted to $f$, i.e. a stratification
\[
X = \bigsqcup X_\la
\]
of $X$ into a finite disjoint union of locally closed, connected and
smooth subvarieties such that the induced map $f : f^{-1}(X_\la) \to
X_\la$ is a topologically locally trivial fibration in (usually
singular) varieties.

By the decomposition theorem of Beilinson, Bernstein, Deligne and
Gabber, $Rf_*\QM_{\widetilde{X}}$ is isomorphic to a direct sum of shifts of
intersection cohomology sheaves on $X$. Let us fix a non-unit $M \in
\ZM$ and consider the ring $\Lambda = \ZM/M\ZM$. We would like to
understand when the decomposition holds for $Rf_*
\Lambda_{\widetilde{X}}$.

For each stratum $X_\la$ and point $x \in X_\la$ we can choose a normal
slice $N$ to the stratum $X_\la$ through $x$. If we set $F :=
f^{-1}(x)$ and $\widetilde{N} := f^{-1}(N)$ we have a commutative
diagram of Cartesian squares:
\[
\xymatrix{ F \ar[r] \ar[d] & \widetilde{N} \ar[r] \ar[d] & \widetilde{X} \ar[d] \\
 \{x \} \ar[r] & N \ar[r] & X }
\]
Set $d := \dim \widetilde{N} = \dim N = \codim (X_\la \subset X)$. The inclusion
$F \into \widetilde{N}$ equips the integral homology of $F$ with an
intersection form (see \cite[\S~3.1]{JMW2})
\begin{equation} \label{eq:if}
IF_\la : H_{d-j}(F ; \ZM) \times H_{d+j}(F; \ZM) \to H_0(\widetilde{N_\la}; \ZM) = \ZM.
\end{equation}

\begin{remark}
  For different points $x, x' \in X_\la$ and normal slices $N, N'$ the pairs $f^{-1}(x) \subset f^{-1}(N)$ and
  $f^{-1}(x') \subset f^{-1}(N')$ are diffeomorphic, though not
  canonically (the isotopy class of diffeomorphism depends on the
  homotopy type of a path from $x$ to $x'$).
\end{remark}

Let us make the following (restrictive) assumptions, which hold
for Schubert varieties and their Bott-Samelson resolutions:
\begin{enumerate}
\item the integral homology $H_*(F; \ZM)$ of all fibres $F$ of $f$ is
  free over $\ZM$;
\item each stratum $X_\la$ is simply connected.
\end{enumerate}

Under these assumptions we have (see \cite[\S~3]{JMW2}):

\begin{thm} The decomposition theorem for $f$
  holds with coefficients in $\Lambda$ if and only if all intersection forms \eqref{eq:if}
  have the same rank over $\QM$ as they do over $\Lambda$.
\end{thm}

The approach of this paper is to calculate these intersection forms in
some special cases. In general this is a difficult task. We now
consider some situations where the job is easier.

Let $F$ and $\widetilde{N}$ be as above. Suppose first that $\dim F <
\frac{1}{2}d$. In this case $H_{d+j}(F) = 0$ for $j \ge 0$. Hence
all intersection forms are zero and the conditions of the theorem are
vacuous. If one has the inequality
\[
\dim f^{-1}(x) < \frac{1}{2} \codim (X_\la \subset X)
\]
for all strata $X_\la$ and $x \in X_\la$ except for those over which $f$ is an isomorphism then
$f$ is \emph{small}. In this case  $Rf_* \Lambda_X[ \dim X] =
\ic(X, \Lambda)$ for any $\Lambda$, which explains why the
decomposition theorem is easy in this case.

\begin{remark}
  By a theorem of Zelevinsky \cite{Zel}, Schubert varieties in
  Grassmannians always admit small resolutions. In particular, the
  stalks and costalks of their 
  integral intersection cohomology complexes are free of $p$-torsion.
\end{remark}

Now suppose that $\dim F = \frac{1}{2}d$. In other words $F \subset
\widetilde{N}$ is half-dimensional and $d$ is the real dimension of $F$. Hence there is only one
intersection form which can be non-zero, namely
\[
H_d(F; \ZM) \times H_d(F; \ZM) \to \ZM.
\]
In this case we have a canonical isomorphism
\[
H_d(F; \ZM) = \bigoplus \ZM[Z]
\]
where the direct sum is over the fundamental classes $[Z]$ of the
irreducible components $Z \subset F$ of maximal dimension.

If the inequality
\[
\dim f^{-1}(x) \le \frac{1}{2} \codim (X_\la \subset X)
\]
holds for all strata $X_\la$ and $x \in X_\la$ then $f$ is called
\emph{semi-small}.  In this case there is only one intersection form
per stratum. However in this case controlling the intersection
forms can be a difficult task.

\begin{ex}
  A simple and rich source of semi-small maps are provided by the
  minimal  resolutions of Kleinian surface singularities $X$ (i.e. quotients $\CM^2 /
  \Gamma$ where $\Gamma \subset SL_2(\CM)$ is a finite subgroup). Here
  $X$ has a unique singular point $0 \in X$ and the 
  exceptional fibre $f^{-1}(0)$ gives a collection of transversely
  intersecting $\PM^1$'s, whose dual graph determines a simply laced
  Dynkin diagram. The intersection form is given by
  the negative of the corresponding Cartan matrix. Hence the decomposition theorem
  is controlled by the determinant of the Cartan matrix. This example has been studied in
  detail by Juteau \cite{decperv}.
\end{ex}

\begin{remark}
  In the case of semi-small maps these forms are non-degenerate and
  even definite (of sign determined by the codimension of the
  strata). This observation is the starting point for de Cataldo and
  Migliorini's Hodge theoretic proof of the decomposition theorem
  \cite{dCM,dCM2}.
\end{remark}

Now assume that the equality $\dim F = d$ holds and that $F$ is
irreducible. Then $H_d(F;\ZM)$ is free of rank one (with basis given by
the fundamental class $[F]$) and the intersection form is a $1 \times
1$-matrix. If $F$ is in addition smooth then we have a diffeomorphism
of pairs 
\[
(F \subset \widetilde{N}) \simto (F \subset N_{\widetilde{N}/F})
\]
where $N_{\widetilde{N}/F}$ is the normal bundle to $F$ in
$\widetilde{N}$. It follows from standard algebraic topology that in
this case the intersection form is given 
$p_! e(N_{\widetilde{N}/F})$ where $p_! : H^{top}(F) \to \ZM$ denotes
the trace map on cohomology and $e(N_{\widetilde{N}/F})$ denotes the
Euler class of the vector bundle $N_{\widetilde{N}/F}$ on $F$.

We refer to this case ($F$ irreducible and smooth) as the \emph{miracle
  situation} because it gives a situation in which the intersection
forms are manageable but non-trivial. After all it is
not difficult to calculate the determinant of a $1 \times 1$-matrix!

\begin{ex} Suppose that $Y$ is a smooth variety such that $Y \subset
  T^*Y$ may be contracted to a point. (That is, there exists a map $f : T^*Y \to X$ which
  is an isomorphism on $T^*Y \setminus Y$ and maps $Y$ to a
  point $x_0 \in X$.) In this
  case $x_0$ is the only singular point in $X$ and $f$ is
  semi-small. Also, as $f^{-1}(x_0) = Y$ we are in the miracle
  situation. The intersection form is given by the Euler class of
  $T^*Y$ which is the $- \chi(Y)$, where $\chi(Y)$ denotes the Euler
  characteristic of $Y$. By the above discussion, the
  decomposition theorem holds with coefficients in $\Lambda$ if and
  only if  the image of $\chi(Y)$  in $\Lambda$ is invertible.

An example of this situation is when $Y = \PM^n$ in which case $X$ may
be realized as the space of rank one matrices in
$\mathfrak{sl}_n(\CM)$ (a minimal nilpotent orbit), see
\cite[\S 3.2]{JMW1}. In this case the intersection form is
$(-\chi(\PM^n)) = (-(n+1))$.
\end{ex}

\section{Groups and Schubert varieties}\label{sec:schubert}

Throughout we work with $G = GL_N(\CM)$, with $T \subset G$ the
maximal torus of diagonal matrices. We denote by $W = S_N$ the Weyl
group of $G$ with simple reflections $S = \{ s_i \}_{i = 1}^{N-1}$ the simple
transpositions. We will often regard $W$ as the subgroup of $G$
 of permutation matrices.

Let $\e_i$ denote the character of $T$ given by $\e_i
(diag(\la_1, \dots, \la_n)) = \la_i$. We let $B$
(resp. $B^-$) denote the subgroup of upper (resp. lower) triangular
matrices. Let $\Phi = \{ \e_i - \e_j \; | \; i \ne j \}$ denote the
 roots, $\Phi^+ := \{ \e_i - \e_j \; | \; i < j \}$ denote the
 positive roots and $\Phi^- := -\Phi^+$ the negative roots. For $1 \le i \le n$ let $\a_i := \e_i - \e_{i+1}$
denote the simple root.

For any $t = s_i \in S$ we denote by $P_t$ the minimal standard parabolic
subgroup with roots $\Phi^+ \cup \{ - \a_i \}$. For any subset $M
\subset S$ we consider the corresponding standard parabolic subgroups
$W_M = \langle t \; | \; t \in M \rangle \subset W$ and
$P_M = \langle P_t \; | \; t \in M \rangle \subset G$. We denote the
corresponding subroot system by $\Phi_M$ and its positive and negative
roots by $\Phi_M^+$ and $\Phi_M^-$.

For $M \subset S$ we have the partial flag variety $G/P_M$. Keeping in
mind that we identify $W$ with permutation matrices we have a natural
map $W \to G/B$ whose image is $(G/B)^T$, the $T$-fixed points on
$G/B$. Simiarly we have a canonical identification
\[
W/W_M \simto (G/P_M)^T.
\]
We will abuse notation and identify a coset $wW_M \in W/W_M$ with the
corresponding fixed point in $G/P_M$. Given any $xW_M \in W/W_M$ we
have a Schubert cell
\[
X_x := B \cdot xP_M/P_M \subset G/P_M
\]
(an affine space) and its closure
\[
\overline{X}_x \subset G/P_M
\]
a Schubert variety.


Let $\z^\vee : \CM^* \to T$ denote a dominant regular cocharacter
(i.e. such that the induced action of $\CM^*$ on the Lie
algebra of the unipotent radical of $B$ has strictly positive
weights). Then the Bia{\l}ynicki-Birula cells on $G/P_M$ coincide with
the Bruhat cells. In other words, for any for any $T$-fixed point
$xW_M$ in $G/P_M$ we have:
\[
X_x = \{ q \in G/P_M \; | \; \lim_{z \to 0} \z^\vee(z)
\cdot q = x \}.
\]
If instead we consider the dual Bia{\l}ynicki-Birula cells we get a
stratification dual to the Bruhat stratification. We have
\[
S_x := B^- \cdot xP_M/P_M = \{ q \in G/P_M \; | \; \lim_{z \to \infty} \z^\vee(z)
\cdot q = x \}.
\]

\section{Bott-Samelson varieties} \label{sec:bs}

Given a sequence $\un{w}:= t_1 t_2\dots t_m$ with $t_i \in S$ consider
the Bott-Samelson variety
\[
BS(\un{w}) := P_{t_1} \times_B P_{t_2} \times_B \dots \times_B P_{t_m}/B
\]
defined as the quotient of $P_{t_1} \times P_{t_2} \times \dots \times
P_{t_m}$ by $B^m$ acting on the right by
\[
(p_1, p_2 \dots, p_m) \cdot (b_1, b_2, \dots, b_m) = (p_1b_1,
b_1^{-1}p_2b_2, \dots ,b_{m-1}^{-1}p_m b_m).
\]
We denote the image of $(p_1, \dots, p_m)$ in $BS(\un{w})$ by $[p_1,
\dots, p_m]$. The Bott-Samelson variety $BS(\un{w})$ is a (left) $B$-variety via  $b \cdot [p_1, p_2,\dots, p_m] := [bp_1,p_2, \dots, p_m]$.

Given any subexpression $\un{e}$ of
$\un{w}$ we have a point $[\un{e}] := [s_1^{e_1}, \dots, s_m^{e_m}] \in
BS(\un{w})$ which is fixed by $T$. For any $t
\in S$ we let $u_t(\lambda)$ denote the root subgroup corresponding to
$-\a_t$. Then for any subexpression $\un{e}$ we have a chart around
$[\un{e}]$ given by
\begin{equation}
  \label{eq:chart}
\CM^m \ni (\la_1, \dots, \la_m) \mapsto [t_1^{e_1}u_{t_1}(\la_1),
t_2^{e_2}u_{t_2}(\la_2), \dots, t_m^{e_m}u_{t_m}(\la_m)] \in
BS(\un{w}).
\end{equation}
We denote this chart by $C_{\underline{e}} \subset
BS(\un{w})$. The charts $C_{\underline{e}}$ cover $BS(\un{w})$ as we run over all subexpressions
$\un{e}$. Moreover, using the relation $\g u_t(\lambda) \g^{-1} =
u_t(\a_t(\g)^{-1}\lambda)$ one checks easily that the
$T$-action on $C_{\underline{e}}$ is linear, with weights
\begin{equation}\label{eq:Twts}
(t_1^{e_1}(-\alpha_{t_1}), (t_1^{e_1}t_2^{e_2})(-\alpha_{t_2}), \dots,
(t_1^{e_1} \dots t_m^{e_m})(-\alpha_{t_m}) ).
\end{equation}
In particular the set $\{ [\un{e}] \; | \; \un{e}\text{ a subexpression 
  of } \un{w}\}$ coincides with the set of $T$-fixed points on $BS(\un{w})$.

For any subsequence $\un{e}$ of $\un{w}$ we have a closed subvariety
\[
BS(\un{e}) := \{ [p_1, \dots, p_m] \; | \; p_i = 1\text{ if }e_i = 0
\} \subset BS(\un{w}).
\]
For example $BS(11\dots1) = BS(\un{w})$ and $BS(00\dots 0) = \pt$.
It is easy to see that $BS(\un{e})$ is isomorphic to the Bott-Samelson
variety $BS(\un{z})$ where $\un{z}$ is the
expression obtained from $s_1^{e_1} \dots s_m^{e_m}$ by deleting all
occurrences of the identity.

We denote by $C_{\un{e}}^+ \subset BS(\un{w})$ the Bia{\l}ynicki-Birula
cell corresponding to the $T$-fixed poing $[\un{e}]$ and cocharacter
$\z^\vee$. That is
\[
C_{\un{e}}^+ := \{ x \in BS(\un{w}) \; | \; \lim_{z \to 0} \z^\vee(z)
\cdot x = [\un{e}] \}.
\]
Because $C_{\un{e}} \subset BS(\un{w})$ is open and $T$-stable we have
\[
C_{\un{e}}^+ \subset C_{\un{e}}.
\]
Moreover, from the above calculation of $T$-weights it follows that
\begin{equation} \label{eq:bsbb}
C_{\un{e}}^+ = \{ (\la_i) \in C_{\un{e}} \; | \; \la_i = 0 \text{ if }
(t_1^{e_1}\dots t_i^{e_i})(-\alpha_{t_i}) \in \Phi^- \}.
\end{equation}
(We use the above identification of $C_{\un{e}}$ with $\CM^m$.)

For any $M \subset S$, the multiplication map induces a proper morphism of varieties
\[
f : BS(\un{w}) \to G/P_M.
\]
If $w$ is minimal in its coset $wW_M$ and if $\un{w}$ is a reduced
expression for $w$ then $f$ is an isomorphism over the Schubert cell
$X_w \subset G/P_M$. In particular, in this case $f$ gives a resolution of
singularities of the Schubert variety $\overline{X}_w$. The following
easy lemma will be useful later:

\begin{lem} \label{lem:fixedpoints}
  The $T$-fixed points in the fibre $f^{-1}(xW_M)$ are given by
\[
\{ [\un{e}] \; | \; \text{subexpressions $\un{e}$ of $\un{w}$ with }(\un{w}^{\un{e}})W_M = xW_M \}.
\]
\end{lem}

\section{Combinatorics of reduced expressions}
\label{sec:comb}

In this section we define the reduced expression which determines both the
Schubert variety and the Bott-Samelson resolution occurring in Theorem
\ref{thm:main}. We also establish two combinatorial lemmas involving
this subexpression. Their statements are essentially copied
from \cite{WT}.

\begin{remark}
  The reduced expression combinatorics involved in \cite{WT} is
  slightly more general than that considered here. We hope that this
  makes the treatment below easier to follow.
\end{remark}

Recall the ring $R = \ZM[\e_1, \dots, \e_n]$, the divided difference
operators $\partial_i : R \to R$ and their composites $\partial_w : R
\to R$ for $w \in S_n$ from the introduction. Fix an expression of the form:
\[
C = \partial_{w_m}(\e_n^{a_m}
\dots \partial_{w_2}(\e_n^{a_2}\partial_{w_1}(\e_n^{a_1})) \dots).
\]
The following assumption will be in place for the rest of this paper:
\begin{equation}
    \label{eq:non-zero}
  0 \ne C \in \ZM.
\end{equation}
We set $a := \sum_{i = 1}^m a_i$. Because $\e_i$ has degree 2 and $\partial_w$ has
degree $-2\ell(w)$, \eqref{eq:non-zero} is equivalent to the assumptions:
\begin{gather}
  \label{eq:equal}
    a = \sum_{i = 1}^m \ell(w_i), \\
C \ne 0.
\end{gather}

Because $[\partial_j, \e_n ] = 0$ for $j \ne
n-1$ we may and do assume that $w_i$ is minimal in $S_n / \langle s_1, \dots,
s_{n-2} \rangle$ for all $i$. It follows that each $w_i$ has a unique
reduced expression. It has the form
\[
\un{w}_i = s_{k_i} s_{k_i+1} \dots s_{n-1}
\]
for some $1 \le k_i \le n-1$.

We work in $S_N$ where $N = n + a$. Consider the subsets $M = \{ s_1,
\dots, s_{n-1}\}$ and $A = \{ s_{n+1}, \dots, s_{n+a-1}\}$.
That is, we divide the nodes of our Coxeter diagram as follows:
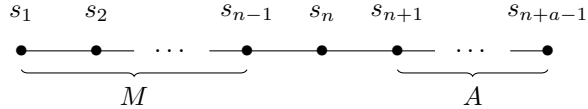
\[
\begin{tikzpicture}
\foreach \x in {0,1,3,4,5,7} \node at (\x,0) {$\bullet$};
\foreach \x/\la in {0/1,1/2,3/n-1,4/n,5/n+1,7/n+a-1} \node at (\x,0.5) {$s_{\la}$};
\foreach \x in {2,6} \node at (\x,0) {$\dots$};
\draw (0,0) -- (1.5,0);
\draw (2.5,0) -- (5.5,0);
\draw (6.5,0) -- (7,0);
\def\y{-0.15}
\draw [ decoration={brace,mirror,raise=0.1cm},decorate] (0,\y) to (3,\y);
\draw [ decoration={brace,mirror,raise=0.1cm},decorate] (5,\y) to
(7,\y);
\node at (1.5,4*\y) {$M$};
\node at (6,4*\y) {$A$};
\end{tikzpicture}
\]
Let $W_M$ and $W_A$ denote the corresponding parabolic subgroups. The
simple reflection $s_n$ plays a special role, as will become clear shortly.

Consider
\[
\un{x} = \un{w}_{m} \un{z}_m \dots \un{w}_2 \un{z}_2 \un{w}_1 \un{z}_1
\]
where
\begin{align*}
  \un{z}_1 &= (s_ns_{n+1} \dots s_{n+a_1 - 1}) \dots (s_n s_{n+1}) (s_n),\\
  \un{z}_2 &= (s_ns_{n+1} \dots s_{n+a_1 + a_2 - 1}) \dots (s_ns_{n+1} \dots
  s_{n+a_1+1}) (s_ns_{n+1} \dots s_{n+a_1}),\\
& \qquad \vdots\\
\un{z}_m &= (s_ns_{n+1} \dots s_{b-1}) \dots (s_ns_{n+1} \dots s_{n+a-a_m+1})
(s_ns_{n+1}\dots s_{n+a-a_m}).
\end{align*}
We denote by $\un{z}'_1,\un{z}'_2,\dots,\un{z}'_m$ the similar reduced
expressions with all occurrences of $s_n$ deleted:
\begin{align*}
  \un{z}'_1 &= (s_{n+1} \dots s_{n+a_1 - 1}) \dots (s_{n+1}),\\
& \qquad \vdots\\
\un{z}'_m &= (s_{n+1} \dots s_{a-1}) \dots (s_{n+1} \dots s_{n+a-a_m+1})
(s_{n+1}\dots s_{n+a-a_m}).
\end{align*}

\begin{remark}
 The expression $\un{z}_m  \dots \un{z}_2 \un{z}_1$ is a reduced expression for $w_{\{s_n\} \cup A}$. Similarly, $\un{z}'_m  \dots \un{z}'_2 \un{z}'_1$ is a reduced expression for $w_A$.
\end{remark}

\begin{ex} \label{ex:x} We give an example of the expression $\un{x}$. Let $n = 4$
 and consider $\un{w}_i$ defined as follows:\[
\un{w}_1 = s_2s_3, \quad \un{w}_2 = s_3,\quad \un{w}_3 = s_2s_3,\quad \un{w}_4 =
s_1s_2s_3 = \un{w}_5.
\]
Take $a_1 =3$, $a_2 = 2$, $a_3 = 2$, $a_4 = 2$, $a_5 =
2$, so that $a = 3 + 2 + 2 + 2 + 2 = 11$ and $N = 4 + 11 = 15$.
Then we may depict $\un{x}$ via the following string diagram:
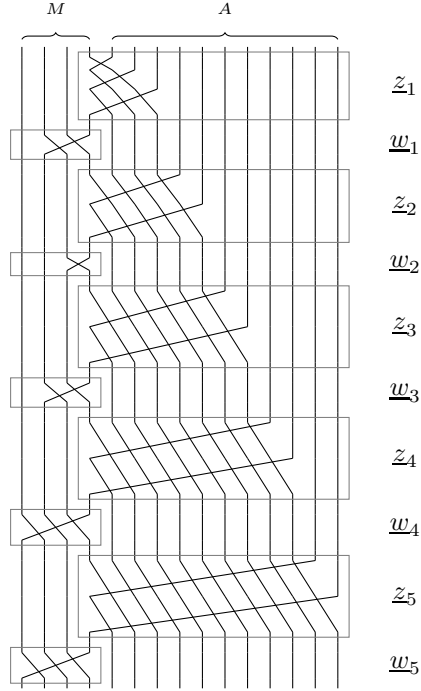
\begin{equation*}
\begin{array}{c}
\begin{tikzpicture}[xscale=0.3,yscale=0.17]
\node at (2.5,13) {\tiny $M$};
\node at (10,13) {\tiny $A$};
\draw [ decoration={brace,raise=0.1cm},decorate] (1,10) to (4,10);
\draw [ decoration={brace,raise=0.1cm},decorate] (5,10) to (15,10);
\draw (1,10.0000) -- (1,9.20000);
\draw (2,10.0000) -- (2,9.20000);
\draw (3,10.0000) -- (3,9.20000);
\draw (4,10.0000) -- (4,9.20000);
\draw (5,10.0000) -- (5,9.20000);
\draw (6,10.0000) -- (6,9.20000);
\draw (7,10.0000) -- (7,9.20000);
\draw (8,10.0000) -- (8,9.20000);
\draw (9,10.0000) -- (9,9.20000);
\draw (10,10.0000) -- (10,9.20000);
\draw (11,10.0000) -- (11,9.20000);
\draw (12,10.0000) -- (12,9.20000);
\draw (13,10.0000) -- (13,9.20000);
\draw (14,10.0000) -- (14,9.20000);
\draw (15,10.0000) -- (15,9.20000);
\draw (1,9.20000) -- (1,8.20000);
\draw (2,9.20000) -- (2,8.20000);
\draw (3,9.20000) -- (3,8.20000);
\draw (6,9.20000) -- (6,8.20000);
\draw (7,9.20000) -- (7,8.20000);
\draw (8,9.20000) -- (8,8.20000);
\draw (9,9.20000) -- (9,8.20000);
\draw (10,9.20000) -- (10,8.20000);
\draw (11,9.20000) -- (11,8.20000);
\draw (12,9.20000) -- (12,8.20000);
\draw (13,9.20000) -- (13,8.20000);
\draw (14,9.20000) -- (14,8.20000);
\draw (15,9.20000) -- (15,8.20000);
\draw (4,9.20000) to (5,8.20000);
\draw (5,9.20000) to (4,8.20000);
\draw (1,8.20000) -- (1,6.70000);
\draw (2,8.20000) -- (2,6.70000);
\draw (3,8.20000) -- (3,6.70000);
\draw (7,8.20000) -- (7,6.70000);
\draw (8,8.20000) -- (8,6.70000);
\draw (9,8.20000) -- (9,6.70000);
\draw (10,8.20000) -- (10,6.70000);
\draw (11,8.20000) -- (11,6.70000);
\draw (12,8.20000) -- (12,6.70000);
\draw (13,8.20000) -- (13,6.70000);
\draw (14,8.20000) -- (14,6.70000);
\draw (15,8.20000) -- (15,6.70000);
\draw (4,8.20000) to (5,6.70000);
\draw (5,8.20000) to (6,6.70000);
\draw (6,8.20000) to (4,6.70000);
\draw (1,6.70000) -- (1,4.70000);
\draw (2,6.70000) -- (2,4.70000);
\draw (3,6.70000) -- (3,4.70000);
\draw (8,6.70000) -- (8,4.70000);
\draw (9,6.70000) -- (9,4.70000);
\draw (10,6.70000) -- (10,4.70000);
\draw (11,6.70000) -- (11,4.70000);
\draw (12,6.70000) -- (12,4.70000);
\draw (13,6.70000) -- (13,4.70000);
\draw (14,6.70000) -- (14,4.70000);
\draw (15,6.70000) -- (15,4.70000);
\draw (4,6.70000) to (5,4.70000);
\draw (5,6.70000) to (6,4.70000);
\draw (6,6.70000) to (7,4.70000);
\draw (7,6.70000) to (4,4.70000);
\draw (1,4.70000) -- (1,3.90000);
\draw (2,4.70000) -- (2,3.90000);
\draw (3,4.70000) -- (3,3.90000);
\draw (4,4.70000) -- (4,3.90000);
\draw (5,4.70000) -- (5,3.90000);
\draw (6,4.70000) -- (6,3.90000);
\draw (7,4.70000) -- (7,3.90000);
\draw (8,4.70000) -- (8,3.90000);
\draw (9,4.70000) -- (9,3.90000);
\draw (10,4.70000) -- (10,3.90000);
\draw (11,4.70000) -- (11,3.90000);
\draw (12,4.70000) -- (12,3.90000);
\draw (13,4.70000) -- (13,3.90000);
\draw (14,4.70000) -- (14,3.90000);
\draw (15,4.70000) -- (15,3.90000);
\draw[color=gray] (3.5000000,9.60001) rectangle (15.500000,4.30000);
\node at (18, 6.9500000) {$\un{z}_1$};
\draw (1,3.90000) -- (1,3.10000);
\draw (2,3.90000) -- (2,3.10000);
\draw (3,3.90000) -- (3,3.10000);
\draw (4,3.90000) -- (4,3.10000);
\draw (5,3.90000) -- (5,3.10000);
\draw (6,3.90000) -- (6,3.10000);
\draw (7,3.90000) -- (7,3.10000);
\draw (8,3.90000) -- (8,3.10000);
\draw (9,3.90000) -- (9,3.10000);
\draw (10,3.90000) -- (10,3.10000);
\draw (11,3.90000) -- (11,3.10000);
\draw (12,3.90000) -- (12,3.10000);
\draw (13,3.90000) -- (13,3.10000);
\draw (14,3.90000) -- (14,3.10000);
\draw (15,3.90000) -- (15,3.10000);
\draw (1,3.10000) -- (1,1.60000);
\draw (5,3.10000) -- (5,1.60000);
\draw (6,3.10000) -- (6,1.60000);
\draw (7,3.10000) -- (7,1.60000);
\draw (8,3.10000) -- (8,1.60000);
\draw (9,3.10000) -- (9,1.60000);
\draw (10,3.10000) -- (10,1.60000);
\draw (11,3.10000) -- (11,1.60000);
\draw (12,3.10000) -- (12,1.60000);
\draw (13,3.10000) -- (13,1.60000);
\draw (14,3.10000) -- (14,1.60000);
\draw (15,3.10000) -- (15,1.60000);
\draw (2,3.10000) to (3,1.60000);
\draw (3,3.10000) to (4,1.60000);
\draw (4,3.10000) to (2,1.60000);
\draw (1,1.60000) -- (1,0.800000);
\draw (2,1.60000) -- (2,0.800000);
\draw (3,1.60000) -- (3,0.800000);
\draw (4,1.60000) -- (4,0.800000);
\draw (5,1.60000) -- (5,0.800000);
\draw (6,1.60000) -- (6,0.800000);
\draw (7,1.60000) -- (7,0.800000);
\draw (8,1.60000) -- (8,0.800000);
\draw (9,1.60000) -- (9,0.800000);
\draw (10,1.60000) -- (10,0.800000);
\draw (11,1.60000) -- (11,0.800000);
\draw (12,1.60000) -- (12,0.800000);
\draw (13,1.60000) -- (13,0.800000);
\draw (14,1.60000) -- (14,0.800000);
\draw (15,1.60000) -- (15,0.800000);
\draw[color=gray] (0.50000000,3.50000) rectangle (4.5000000,1.20000);
\node at (18, 2.3500000) {$\un{w}_1$};
\draw (1,0.800000) -- (1,0);
\draw (2,0.800000) -- (2,0);
\draw (3,0.800000) -- (3,0);
\draw (4,0.800000) -- (4,0);
\draw (5,0.800000) -- (5,0);
\draw (6,0.800000) -- (6,0);
\draw (7,0.800000) -- (7,0);
\draw (8,0.800000) -- (8,0);
\draw (9,0.800000) -- (9,0);
\draw (10,0.800000) -- (10,0);
\draw (11,0.800000) -- (11,0);
\draw (12,0.800000) -- (12,0);
\draw (13,0.800000) -- (13,0);
\draw (14,0.800000) -- (14,0);
\draw (15,0.800000) -- (15,0);
\draw (1,0) -- (1,-2.30000);
\draw (2,0) -- (2,-2.30000);
\draw (3,0) -- (3,-2.30000);
\draw (9,0) -- (9,-2.30000);
\draw (10,0) -- (10,-2.30000);
\draw (11,0) -- (11,-2.30000);
\draw (12,0) -- (12,-2.30000);
\draw (13,0) -- (13,-2.30000);
\draw (14,0) -- (14,-2.30000);
\draw (15,0) -- (15,-2.30000);
\draw (4,0) to (5,-2.30000);
\draw (5,0) to (6,-2.30000);
\draw (6,0) to (7,-2.30000);
\draw (7,0) to (8,-2.30000);
\draw (8,0) to (4,-2.30000);
\draw (1,-2.30000) -- (1,-4.89999);
\draw (2,-2.30000) -- (2,-4.89999);
\draw (3,-2.30000) -- (3,-4.89999);
\draw (10,-2.30000) -- (10,-4.89999);
\draw (11,-2.30000) -- (11,-4.89999);
\draw (12,-2.30000) -- (12,-4.89999);
\draw (13,-2.30000) -- (13,-4.89999);
\draw (14,-2.30000) -- (14,-4.89999);
\draw (15,-2.30000) -- (15,-4.89999);
\draw (4,-2.30000) to (5,-4.89999);
\draw (5,-2.30000) to (6,-4.89999);
\draw (6,-2.30000) to (7,-4.89999);
\draw (7,-2.30000) to (8,-4.89999);
\draw (8,-2.30000) to (9,-4.89999);
\draw (9,-2.30000) to (4,-4.89999);
\draw (1,-4.90000) -- (1,-5.70000);
\draw (2,-4.90000) -- (2,-5.70000);
\draw (3,-4.90000) -- (3,-5.70000);
\draw (4,-4.90000) -- (4,-5.70000);
\draw (5,-4.90000) -- (5,-5.70000);
\draw (6,-4.90000) -- (6,-5.70000);
\draw (7,-4.90000) -- (7,-5.70000);
\draw (8,-4.90000) -- (8,-5.70000);
\draw (9,-4.90000) -- (9,-5.70000);
\draw (10,-4.90000) -- (10,-5.70000);
\draw (11,-4.90000) -- (11,-5.70000);
\draw (12,-4.90000) -- (12,-5.70000);
\draw (13,-4.90000) -- (13,-5.70000);
\draw (14,-4.90000) -- (14,-5.70000);
\draw (15,-4.90000) -- (15,-5.70000);
\draw[color=gray] (3.5000000,0.400000) rectangle (15.500000,-5.30000);
\node at (18, -2.4500000) {$\un{z}_2$};
\draw (1,-5.70000) -- (1,-6.50000);
\draw (2,-5.70000) -- (2,-6.50000);
\draw (3,-5.70000) -- (3,-6.50000);
\draw (4,-5.70000) -- (4,-6.50000);
\draw (5,-5.70000) -- (5,-6.50000);
\draw (6,-5.70000) -- (6,-6.50000);
\draw (7,-5.70000) -- (7,-6.50000);
\draw (8,-5.70000) -- (8,-6.50000);
\draw (9,-5.70000) -- (9,-6.50000);
\draw (10,-5.70000) -- (10,-6.50000);
\draw (11,-5.70000) -- (11,-6.50000);
\draw (12,-5.70000) -- (12,-6.50000);
\draw (13,-5.70000) -- (13,-6.50000);
\draw (14,-5.70000) -- (14,-6.50000);
\draw (15,-5.70000) -- (15,-6.50000);
\draw (1,-6.50000) -- (1,-7.50000);
\draw (2,-6.50000) -- (2,-7.50000);
\draw (5,-6.50000) -- (5,-7.50000);
\draw (6,-6.50000) -- (6,-7.50000);
\draw (7,-6.50000) -- (7,-7.50000);
\draw (8,-6.50000) -- (8,-7.50000);
\draw (9,-6.50000) -- (9,-7.50000);
\draw (10,-6.50000) -- (10,-7.50000);
\draw (11,-6.50000) -- (11,-7.50000);
\draw (12,-6.50000) -- (12,-7.50000);
\draw (13,-6.50000) -- (13,-7.50000);
\draw (14,-6.50000) -- (14,-7.50000);
\draw (15,-6.50000) -- (15,-7.50000);
\draw (3,-6.50000) to (4,-7.50000);
\draw (4,-6.50000) to (3,-7.50000);
\draw (1,-7.50000) -- (1,-8.30000);
\draw (2,-7.50000) -- (2,-8.30000);
\draw (3,-7.50000) -- (3,-8.30000);
\draw (4,-7.50000) -- (4,-8.30000);
\draw (5,-7.50000) -- (5,-8.30000);
\draw (6,-7.50000) -- (6,-8.30000);
\draw (7,-7.50000) -- (7,-8.30000);
\draw (8,-7.50000) -- (8,-8.30000);
\draw (9,-7.50000) -- (9,-8.30000);
\draw (10,-7.50000) -- (10,-8.30000);
\draw (11,-7.50000) -- (11,-8.30000);
\draw (12,-7.50000) -- (12,-8.30000);
\draw (13,-7.50000) -- (13,-8.30000);
\draw (14,-7.50000) -- (14,-8.30000);
\draw (15,-7.50000) -- (15,-8.30000);
\draw[color=gray] (0.50000000,-6.10000) rectangle (4.5000000,-7.90000);
\node at (18, -7.0000000) {$\un{w}_2$};
\draw (1,-8.30000) -- (1,-9.10001);
\draw (2,-8.30000) -- (2,-9.10001);
\draw (3,-8.30000) -- (3,-9.10001);
\draw (4,-8.30000) -- (4,-9.10001);
\draw (5,-8.30000) -- (5,-9.10001);
\draw (6,-8.30000) -- (6,-9.10001);
\draw (7,-8.30000) -- (7,-9.10001);
\draw (8,-8.30000) -- (8,-9.10001);
\draw (9,-8.30000) -- (9,-9.10001);
\draw (10,-8.30000) -- (10,-9.10001);
\draw (11,-8.30000) -- (11,-9.10001);
\draw (12,-8.30000) -- (12,-9.10001);
\draw (13,-8.30000) -- (13,-9.10001);
\draw (14,-8.30000) -- (14,-9.10001);
\draw (15,-8.30000) -- (15,-9.10001);
\draw (1,-9.10001) -- (1,-11.9000);
\draw (2,-9.10001) -- (2,-11.9000);
\draw (3,-9.10001) -- (3,-11.9000);
\draw (11,-9.10001) -- (11,-11.9000);
\draw (12,-9.10001) -- (12,-11.9000);
\draw (13,-9.10001) -- (13,-11.9000);
\draw (14,-9.10001) -- (14,-11.9000);
\draw (15,-9.10001) -- (15,-11.9000);
\draw (4,-9.10001) to (5,-11.9000);
\draw (5,-9.10001) to (6,-11.9000);
\draw (6,-9.10001) to (7,-11.9000);
\draw (7,-9.10001) to (8,-11.9000);
\draw (8,-9.10001) to (9,-11.9000);
\draw (9,-9.10001) to (10,-11.9000);
\draw (10,-9.10001) to (4,-11.9000);
\draw (1,-11.9000) -- (1,-14.7000);
\draw (2,-11.9000) -- (2,-14.7000);
\draw (3,-11.9000) -- (3,-14.7000);
\draw (12,-11.9000) -- (12,-14.7000);
\draw (13,-11.9000) -- (13,-14.7000);
\draw (14,-11.9000) -- (14,-14.7000);
\draw (15,-11.9000) -- (15,-14.7000);
\draw (4,-11.9000) to (5,-14.7000);
\draw (5,-11.9000) to (6,-14.7000);
\draw (6,-11.9000) to (7,-14.7000);
\draw (7,-11.9000) to (8,-14.7000);
\draw (8,-11.9000) to (9,-14.7000);
\draw (9,-11.9000) to (10,-14.7000);
\draw (10,-11.9000) to (11,-14.7000);
\draw (11,-11.9000) to (4,-14.7000);
\draw (1,-14.7000) -- (1,-15.5000);
\draw (2,-14.7000) -- (2,-15.5000);
\draw (3,-14.7000) -- (3,-15.5000);
\draw (4,-14.7000) -- (4,-15.5000);
\draw (5,-14.7000) -- (5,-15.5000);
\draw (6,-14.7000) -- (6,-15.5000);
\draw (7,-14.7000) -- (7,-15.5000);
\draw (8,-14.7000) -- (8,-15.5000);
\draw (9,-14.7000) -- (9,-15.5000);
\draw (10,-14.7000) -- (10,-15.5000);
\draw (11,-14.7000) -- (11,-15.5000);
\draw (12,-14.7000) -- (12,-15.5000);
\draw (13,-14.7000) -- (13,-15.5000);
\draw (14,-14.7000) -- (14,-15.5000);
\draw (15,-14.7000) -- (15,-15.5000);
\draw[color=gray] (3.5000000,-8.70000) rectangle (15.500000,-15.1000);
\node at (18, -11.900000) {$\un{z}_3$};
\draw (1,-15.5000) -- (1,-16.3000);
\draw (2,-15.5000) -- (2,-16.3000);
\draw (3,-15.5000) -- (3,-16.3000);
\draw (4,-15.5000) -- (4,-16.3000);
\draw (5,-15.5000) -- (5,-16.3000);
\draw (6,-15.5000) -- (6,-16.3000);
\draw (7,-15.5000) -- (7,-16.3000);
\draw (8,-15.5000) -- (8,-16.3000);
\draw (9,-15.5000) -- (9,-16.3000);
\draw (10,-15.5000) -- (10,-16.3000);
\draw (11,-15.5000) -- (11,-16.3000);
\draw (12,-15.5000) -- (12,-16.3000);
\draw (13,-15.5000) -- (13,-16.3000);
\draw (14,-15.5000) -- (14,-16.3000);
\draw (15,-15.5000) -- (15,-16.3000);
\draw (1,-16.3000) -- (1,-17.8000);
\draw (5,-16.3000) -- (5,-17.8000);
\draw (6,-16.3000) -- (6,-17.8000);
\draw (7,-16.3000) -- (7,-17.8000);
\draw (8,-16.3000) -- (8,-17.8000);
\draw (9,-16.3000) -- (9,-17.8000);
\draw (10,-16.3000) -- (10,-17.8000);
\draw (11,-16.3000) -- (11,-17.8000);
\draw (12,-16.3000) -- (12,-17.8000);
\draw (13,-16.3000) -- (13,-17.8000);
\draw (14,-16.3000) -- (14,-17.8000);
\draw (15,-16.3000) -- (15,-17.8000);
\draw (2,-16.3000) to (3,-17.8000);
\draw (3,-16.3000) to (4,-17.8000);
\draw (4,-16.3000) to (2,-17.8000);
\draw (1,-17.8000) -- (1,-18.6000);
\draw (2,-17.8000) -- (2,-18.6000);
\draw (3,-17.8000) -- (3,-18.6000);
\draw (4,-17.8000) -- (4,-18.6000);
\draw (5,-17.8000) -- (5,-18.6000);
\draw (6,-17.8000) -- (6,-18.6000);
\draw (7,-17.8000) -- (7,-18.6000);
\draw (8,-17.8000) -- (8,-18.6000);
\draw (9,-17.8000) -- (9,-18.6000);
\draw (10,-17.8000) -- (10,-18.6000);
\draw (11,-17.8000) -- (11,-18.6000);
\draw (12,-17.8000) -- (12,-18.6000);
\draw (13,-17.8000) -- (13,-18.6000);
\draw (14,-17.8000) -- (14,-18.6000);
\draw (15,-17.8000) -- (15,-18.6000);
\draw[color=gray] (0.50000000,-15.9000) rectangle (4.5000000,-18.2000);
\node at (18, -17.050000) {$\un{w}_3$};
\draw (1,-18.6000) -- (1,-19.4000);
\draw (2,-18.6000) -- (2,-19.4000);
\draw (3,-18.6000) -- (3,-19.4000);
\draw (4,-18.6000) -- (4,-19.4000);
\draw (5,-18.6000) -- (5,-19.4000);
\draw (6,-18.6000) -- (6,-19.4000);
\draw (7,-18.6000) -- (7,-19.4000);
\draw (8,-18.6000) -- (8,-19.4000);
\draw (9,-18.6000) -- (9,-19.4000);
\draw (10,-18.6000) -- (10,-19.4000);
\draw (11,-18.6000) -- (11,-19.4000);
\draw (12,-18.6000) -- (12,-19.4000);
\draw (13,-18.6000) -- (13,-19.4000);
\draw (14,-18.6000) -- (14,-19.4000);
\draw (15,-18.6000) -- (15,-19.4000);
\draw (1,-19.4000) -- (1,-22.2000);
\draw (2,-19.4000) -- (2,-22.2000);
\draw (3,-19.4000) -- (3,-22.2000);
\draw (13,-19.4000) -- (13,-22.2000);
\draw (14,-19.4000) -- (14,-22.2000);
\draw (15,-19.4000) -- (15,-22.2000);
\draw (4,-19.4000) to (5,-22.2000);
\draw (5,-19.4000) to (6,-22.2000);
\draw (6,-19.4000) to (7,-22.2000);
\draw (7,-19.4000) to (8,-22.2000);
\draw (8,-19.4000) to (9,-22.2000);
\draw (9,-19.4000) to (10,-22.2000);
\draw (10,-19.4000) to (11,-22.2000);
\draw (11,-19.4000) to (12,-22.2000);
\draw (12,-19.4000) to (4,-22.2000);
\draw (1,-22.2000) -- (1,-25.0000);
\draw (2,-22.2000) -- (2,-25.0000);
\draw (3,-22.2000) -- (3,-25.0000);
\draw (14,-22.2000) -- (14,-25.0000);
\draw (15,-22.2000) -- (15,-25.0000);
\draw (4,-22.2000) to (5,-25.0000);
\draw (5,-22.2000) to (6,-25.0000);
\draw (6,-22.2000) to (7,-25.0000);
\draw (7,-22.2000) to (8,-25.0000);
\draw (8,-22.2000) to (9,-25.0000);
\draw (9,-22.2000) to (10,-25.0000);
\draw (10,-22.2000) to (11,-25.0000);
\draw (11,-22.2000) to (12,-25.0000);
\draw (12,-22.2000) to (13,-25.0000);
\draw (13,-22.2000) to (4,-25.0000);
\draw (1,-25.0000) -- (1,-25.8000);
\draw (2,-25.0000) -- (2,-25.8000);
\draw (3,-25.0000) -- (3,-25.8000);
\draw (4,-25.0000) -- (4,-25.8000);
\draw (5,-25.0000) -- (5,-25.8000);
\draw (6,-25.0000) -- (6,-25.8000);
\draw (7,-25.0000) -- (7,-25.8000);
\draw (8,-25.0000) -- (8,-25.8000);
\draw (9,-25.0000) -- (9,-25.8000);
\draw (10,-25.0000) -- (10,-25.8000);
\draw (11,-25.0000) -- (11,-25.8000);
\draw (12,-25.0000) -- (12,-25.8000);
\draw (13,-25.0000) -- (13,-25.8000);
\draw (14,-25.0000) -- (14,-25.8000);
\draw (15,-25.0000) -- (15,-25.8000);
\draw[color=gray] (3.5000000,-19.0000) rectangle (15.500000,-25.4000);
\node at (18, -22.200000) {$\un{z}_4$};
\draw (1,-25.8000) -- (1,-26.6000);
\draw (2,-25.8000) -- (2,-26.6000);
\draw (3,-25.8000) -- (3,-26.6000);
\draw (4,-25.8000) -- (4,-26.6000);
\draw (5,-25.8000) -- (5,-26.6000);
\draw (6,-25.8000) -- (6,-26.6000);
\draw (7,-25.8000) -- (7,-26.6000);
\draw (8,-25.8000) -- (8,-26.6000);
\draw (9,-25.8000) -- (9,-26.6000);
\draw (10,-25.8000) -- (10,-26.6000);
\draw (11,-25.8000) -- (11,-26.6000);
\draw (12,-25.8000) -- (12,-26.6000);
\draw (13,-25.8000) -- (13,-26.6000);
\draw (14,-25.8000) -- (14,-26.6000);
\draw (15,-25.8000) -- (15,-26.6000);
\draw (5,-26.6000) -- (5,-28.6000);
\draw (6,-26.6000) -- (6,-28.6000);
\draw (7,-26.6000) -- (7,-28.6000);
\draw (8,-26.6000) -- (8,-28.6000);
\draw (9,-26.6000) -- (9,-28.6000);
\draw (10,-26.6000) -- (10,-28.6000);
\draw (11,-26.6000) -- (11,-28.6000);
\draw (12,-26.6000) -- (12,-28.6000);
\draw (13,-26.6000) -- (13,-28.6000);
\draw (14,-26.6000) -- (14,-28.6000);
\draw (15,-26.6000) -- (15,-28.6000);
\draw (1,-26.6000) to (2,-28.6000);
\draw (2,-26.6000) to (3,-28.6000);
\draw (3,-26.6000) to (4,-28.6000);
\draw (4,-26.6000) to (1,-28.6000);
\draw (1,-28.6000) -- (1,-29.4000);
\draw (2,-28.6000) -- (2,-29.4000);
\draw (3,-28.6000) -- (3,-29.4000);
\draw (4,-28.6000) -- (4,-29.4000);
\draw (5,-28.6000) -- (5,-29.4000);
\draw (6,-28.6000) -- (6,-29.4000);
\draw (7,-28.6000) -- (7,-29.4000);
\draw (8,-28.6000) -- (8,-29.4000);
\draw (9,-28.6000) -- (9,-29.4000);
\draw (10,-28.6000) -- (10,-29.4000);
\draw (11,-28.6000) -- (11,-29.4000);
\draw (12,-28.6000) -- (12,-29.4000);
\draw (13,-28.6000) -- (13,-29.4000);
\draw (14,-28.6000) -- (14,-29.4000);
\draw (15,-28.6000) -- (15,-29.4000);
\draw[color=gray] (0.50000000,-26.2000) rectangle (4.5000000,-29.0000);
\node at (18, -27.599999) {$\un{w}_4$};
\draw (1,-29.4000) -- (1,-30.2000);
\draw (2,-29.4000) -- (2,-30.2000);
\draw (3,-29.4000) -- (3,-30.2000);
\draw (4,-29.4000) -- (4,-30.2000);
\draw (5,-29.4000) -- (5,-30.2000);
\draw (6,-29.4000) -- (6,-30.2000);
\draw (7,-29.4000) -- (7,-30.2000);
\draw (8,-29.4000) -- (8,-30.2000);
\draw (9,-29.4000) -- (9,-30.2000);
\draw (10,-29.4000) -- (10,-30.2000);
\draw (11,-29.4000) -- (11,-30.2000);
\draw (12,-29.4000) -- (12,-30.2000);
\draw (13,-29.4000) -- (13,-30.2000);
\draw (14,-29.4000) -- (14,-30.2000);
\draw (15,-29.4000) -- (15,-30.2000);
\draw (1,-30.2000) -- (1,-33.0000);
\draw (2,-30.2000) -- (2,-33.0000);
\draw (3,-30.2000) -- (3,-33.0000);
\draw (15,-30.2000) -- (15,-33.0000);
\draw (4,-30.2000) to (5,-33.0000);
\draw (5,-30.2000) to (6,-33.0000);
\draw (6,-30.2000) to (7,-33.0000);
\draw (7,-30.2000) to (8,-33.0000);
\draw (8,-30.2000) to (9,-33.0000);
\draw (9,-30.2000) to (10,-33.0000);
\draw (10,-30.2000) to (11,-33.0000);
\draw (11,-30.2000) to (12,-33.0000);
\draw (12,-30.2000) to (13,-33.0000);
\draw (13,-30.2000) to (14,-33.0000);
\draw (14,-30.2000) to (4,-33.0000);
\draw (1,-33.0000) -- (1,-35.8000);
\draw (2,-33.0000) -- (2,-35.8000);
\draw (3,-33.0000) -- (3,-35.8000);
\draw (4,-33.0000) to (5,-35.8000);
\draw (5,-33.0000) to (6,-35.8000);
\draw (6,-33.0000) to (7,-35.8000);
\draw (7,-33.0000) to (8,-35.8000);
\draw (8,-33.0000) to (9,-35.8000);
\draw (9,-33.0000) to (10,-35.8000);
\draw (10,-33.0000) to (11,-35.8000);
\draw (11,-33.0000) to (12,-35.8000);
\draw (12,-33.0000) to (13,-35.8000);
\draw (13,-33.0000) to (14,-35.8000);
\draw (14,-33.0000) to (15,-35.8000);
\draw (15,-33.0000) to (4,-35.8000);
\draw (1,-35.8000) -- (1,-36.6000);
\draw (2,-35.8000) -- (2,-36.6000);
\draw (3,-35.8000) -- (3,-36.6000);
\draw (4,-35.8000) -- (4,-36.6000);
\draw (5,-35.8000) -- (5,-36.6000);
\draw (6,-35.8000) -- (6,-36.6000);
\draw (7,-35.8000) -- (7,-36.6000);
\draw (8,-35.8000) -- (8,-36.6000);
\draw (9,-35.8000) -- (9,-36.6000);
\draw (10,-35.8000) -- (10,-36.6000);
\draw (11,-35.8000) -- (11,-36.6000);
\draw (12,-35.8000) -- (12,-36.6000);
\draw (13,-35.8000) -- (13,-36.6000);
\draw (14,-35.8000) -- (14,-36.6000);
\draw (15,-35.8000) -- (15,-36.6000);
\draw[color=gray] (3.5000000,-29.8000) rectangle (15.500000,-36.2000);
\node at (18, -33.000000) {$\un{z}_5$};
\draw (1,-36.6000) -- (1,-37.4000);
\draw (2,-36.6000) -- (2,-37.4000);
\draw (3,-36.6000) -- (3,-37.4000);
\draw (4,-36.6000) -- (4,-37.4000);
\draw (5,-36.6000) -- (5,-37.4000);
\draw (6,-36.6000) -- (6,-37.4000);
\draw (7,-36.6000) -- (7,-37.4000);
\draw (8,-36.6000) -- (8,-37.4000);
\draw (9,-36.6000) -- (9,-37.4000);
\draw (10,-36.6000) -- (10,-37.4000);
\draw (11,-36.6000) -- (11,-37.4000);
\draw (12,-36.6000) -- (12,-37.4000);
\draw (13,-36.6000) -- (13,-37.4000);
\draw (14,-36.6000) -- (14,-37.4000);
\draw (15,-36.6000) -- (15,-37.4000);
\draw (5,-37.4000) -- (5,-39.4000);
\draw (6,-37.4000) -- (6,-39.4000);
\draw (7,-37.4000) -- (7,-39.4000);
\draw (8,-37.4000) -- (8,-39.4000);
\draw (9,-37.4000) -- (9,-39.4000);
\draw (10,-37.4000) -- (10,-39.4000);
\draw (11,-37.4000) -- (11,-39.4000);
\draw (12,-37.4000) -- (12,-39.4000);
\draw (13,-37.4000) -- (13,-39.4000);
\draw (14,-37.4000) -- (14,-39.4000);
\draw (15,-37.4000) -- (15,-39.4000);
\draw (1,-37.4000) to (2,-39.4000);
\draw (2,-37.4000) to (3,-39.4000);
\draw (3,-37.4000) to (4,-39.4000);
\draw (4,-37.4000) to (1,-39.4000);
\draw (1,-39.4000) -- (1,-40.2000);
\draw (2,-39.4000) -- (2,-40.2000);
\draw (3,-39.4000) -- (3,-40.2000);
\draw (4,-39.4000) -- (4,-40.2000);
\draw (5,-39.4000) -- (5,-40.2000);
\draw (6,-39.4000) -- (6,-40.2000);
\draw (7,-39.4000) -- (7,-40.2000);
\draw (8,-39.4000) -- (8,-40.2000);
\draw (9,-39.4000) -- (9,-40.2000);
\draw (10,-39.4000) -- (10,-40.2000);
\draw (11,-39.4000) -- (11,-40.2000);
\draw (12,-39.4000) -- (12,-40.2000);
\draw (13,-39.4000) -- (13,-40.2000);
\draw (14,-39.4000) -- (14,-40.2000);
\draw (15,-39.4000) -- (15,-40.2000);
\draw[color=gray] (0.50000000,-37.0000) rectangle (4.5000000,-39.8000);
\node at (18, -38.400000) {$\un{w}_5$};
\end{tikzpicture}
\end{array}
\end{equation*}
\end{ex}

Let $x$ denote the element of $W$ expressed by $\un{x}$. The following
lemma follows by careful consideration of a string diagram depicting
$\un{x}$ (see Example \ref{ex:x} above):

\begin{lem} \label{lem:xprop}
  \begin{enumerate}
  \item $\un{x}$ is a reduced expression for $x$;
  \item $x$ is minimal in its coset $xW_M$.
  \end{enumerate}
\end{lem}

Write $\un{x} = t_1 \dots t_\ell$.

\begin{lem} \label{lem:sub}
Any subsequence $\un{e}$ of $\un{x}$ with $\un{x}^{\un{e}} \in w_AW_M$
has $\e_i = 1$ if $t_i \in A$ and $\e_i = 0$ if $t_i = s_n$.
\end{lem}

\begin{proof}This is (a special case of) Lemma 5.6 in \cite{WT}.
We give an idea of the proof:  We have already remarked that $\un{z}'_m  \dots \un{z}'_2 \un{z}'_1$ is
  a reduced expression for $\un{w}_A$. In particular, to achieve
  $\un{x}^{\un{e}} \in w_AW_M$ we must have $\e_i = 1$ for every $i$ with $t_i \in
  A$. Now, if $\e_i = 1$ for some $i$ with $t_i = s_n$ then it is
  impossible for $\un{x}^{\un{e}}$ to belong to $W_{M \cup A}$, as
  is seen by considering the string diagram of $\un{x}$. The
  result follows.
\end{proof}

\section{Geometry} \label{sec:geometry}

\def\BigBS{\textrm{BigBS}}

We keep the notation from the previous section. Let
$G = GL_{N}(\CM)$ and $P_M$ denote the standard parabolic subgroup
corresponding to $M = \{ s_1 \dots, s_{n-1}\}$.

Consider the Bott-Samelson variety associated to the expression
$\un{x} = t_1 \dots t_\ell$ defined in the previous section:
\[
\textrm{BigBS} := BS(\un{x}) = P_{t_1} \times_B P_{t_2} \times_B \dots \times_B P_{t_\ell} \times_B P_{M}/P_M.
\]
It follows from Lemma \ref{lem:xprop} that  multiplication  induces a resolution of singularities
\[
f : \BigBS \to X_x 
\]
where $X_x$ denotes the Schubert variety $X_x :=
\overline{BxP_M/P_M}\subset G/P_M$.

Consider the closed subvariety
\[
F := \{ [g_1, \dots, g_\ell] \in \BigBS \; | \; g_i = 1 \text{ if
    } t_i=s_n, g_i = t_i \text{ if }t_i\in A \} \subset \BigBS.
\]
(More precisely, we consider the image of the corresponding subset in
$P_{t_1} \times P_{t_2} \times \dots \times P_{t_\ell}$ in $\BigBS$.) 

\begin{lem} \label{lem:fibre}
  $F = f^{-1}(w_AP_M/P_M)$.
\end{lem}

\begin{proof}
  It is clear that $F \subset f^{-1}(w_AP_M/P_M)$. It remains to show
  the reverse inclusion. Consider the subexpression $\un{g} :=
  g_1 \dots g_\ell$ where $g_i = 1$ if $t_i \in M \cup A$ and $g_i =
  0$ if $t_i = s_n$. We first claim that
\begin{equation}
  \label{eq:incl}
f^{-1}(w_AP_M/P_M) \subset BS(\un{g}).
  \end{equation}
(The subvariety $BS(\un{g})$ was defined in \S\ref{sec:bs}.)
Suppose that $q \in f^{-1}(w_AP_M/P_M)$. Recall our dominant regular
cocharacter $\z^\vee$ from earlier. Then $\lim_{z \to 0} \z^\vee(z)
\cdot q \in f^{-1}(w_AP_M/P_M)$ and is a $T$-fixed point. Hence $\lim_{z \to 0} \z^\vee(z)
\cdot q = [\un{e}]$ for some subexpression $\un{e}$ of $\un{x}$, and
$q \in C_{\un{e}}^+$. Now combining Lemma \ref{lem:fixedpoints} and
Lemma \ref{lem:sub}, we see that $\un{e}$
satisfies $e_i = 0$ if $t_i = s_n$ and $e_i = 1$ if $t_i \in A$.
It
follows that $C^+_{\un{e}} \subset BS(\un{g})$ by \eqref{eq:bsbb}. (We use that if $\a \in \Phi^-$ is not in $\Phi_{M \cup A}$, then
$w(\a) \in \Phi^-$ for all $w \in W_{M \cup A}$.) Now
\eqref{eq:incl} follows.

Because the sets $M$ and $A$ are disconnected in the Dynkin diagram we have an isomorphism
\begin{equation} \label{eq:splitting}
BS(\un{g}) \simto BS(\un{w}_m\un{w}_{m-1} \dots \un{w}_1) 
\times BS(\un{z}_m'\un{z}_{m-1}'  \dots \un{z}_1')
\end{equation}
which commutes with the multiplication map $f$. (The expressions
$\un{z}_i'$ were defined in the previous section.) On the first factor
of the right hand side of \eqref{eq:splitting} the multiplication map
is the projection to the base
point $P_M/P_M
\in G/P_M$. Now
$\un{z}_m'\un{z}_{m-1}'  \dots \un{z}_1'$ is a reduced expression for
$w_A$, and hence the fibre of
\[
BS(\un{z}_m'\un{z}_{m-1}'  \dots \un{z}_1') \to G/P_M
\]
over $w_A$ consists only of one point. The result follows.
\end{proof}

The following is easily deduced from \eqref{eq:splitting}.

\begin{cor} \label{cor:bsiso}
  The fibre $F$ is smooth and we have a $T$-equivariant isomorphism
\[
\phi : BS(\un{w}_m \dots \un{w}_1)\simto F.
\]
\end{cor}

In particular, by \eqref{eq:equal}:
\begin{equation}
  \label{eq:4}
  \dim F = a.
\end{equation}

Recall the dual cells $X_{w_A} \subset G/P_M$ and $S_{w_A} \subset
G/P_M$ which are the attracting and repelling sets for the fixed point $w_A \in G/P_M$
and cocharacter $\z^\vee : \CM^* \to T$. Because $S_{w_A}$ is a normal slice to
the stratum $X_{w_A} \subset \overline{X}_x \subset G/P_M$ we
conclude:
\begin{gather}
  \label{eq:smooth}
  f^{-1}(S_{w_A}) \subset \BigBS \text{ is smooth,}\\
\dim f^{-1}(S_{w_A}) = \dim S_{w_A} = \ell(x) - \ell(w_A) = 2a. \label{eq:Sdim}
\end{gather}
In particular we are in the ``miracle situation'': by \eqref{eq:4} and
Lemma \ref{lem:fibre} the fibre of $f$
over $w_AP_M/P_M$ is irreducible, smooth, and half-dimensional inside
$f^{-1}(S_{w_A})$.

By the discussion in \S\ref{sec:what} it follows that in order to 
decide when the decomposition theorem holds at
$w_A$ we need to calculated the self-intersection of $F$ inside
$f^{-1}(S_{w_A})$.
This will be done in the 
next section, and will rely on the following lemma.

\begin{lem} \label{lem:normal}
  Let $[\un{e}]$ be a $T$-fixed point belonging to $F$. Then the
  $T$-weights on the normal bundle to $F \subset  f^{-1}(S_{w_A})$ at $[\un{e}]$  are
\[
\{ t_1^{e_1} t_2^{e_2} \dots t_j^{e_j}(-\a_n) \; | \; 1 \le j \le \ell;
t_j = s_n \}.
\]
\end{lem}

\begin{proof}
Consider the
  chain of inclusions (where $T_{[\un{e}]}$ denotes the tangent space)
\[
T_{[\un{e}]}F \subset T_{[\un{e}]}(f^{-1}(S_{w_A})) \subset T_{[\un{e}]}\BigBS.
\]
Our goal is to calculate the $T$-weights on the normal bundle to $F
\subset f^{-1}(S_{w_A})$ at the $T$-fixed point $[\un{e}]$:
\[
(N_F(f^{-1}S_{w_A}))_{{[\un{e}]}} = T_{[\un{e}]}(f^{-1}S_{w_A})/T_{[\un{e}]}F.
\]
  We work in the chart $C_{\un{e}}$ around $[\un{e}]$. As
  $C_{\un{e}}$ is an affine space with linear $T$-action we have a
  $T$-equivariant identification $C_{\un{e}} =
  T_{[\un{e}]}BS(\un{x})$. Under this identification we claim:
\begin{gather} \label{eq:t1}
  T_{[\un{e}]}F  = \{ (\la_i)_{i = 1}^\ell \; | \; \la_i = 0 \text{ unless } t_i \in
  W_M\}, \\
  T_{[\un{e}]}S_{w_A}  = \{ (\la_i)_{i = 1}^\ell \; | \; \la_i = 0 \text{ if } t_i \in
  W_A\}. \label{eq:t2}
\end{gather}
(We identify $C_{\un{e}} = \CM^\ell$ as in \eqref{eq:chart}.)
The first equality follows from the proof of the previous lemma. For the second
equality notice that if $j$ is such that $t_j = s_n$ (and hence $e_j =
0$) then the curve
\[
c: \g \mapsto [s_1^{e_1}, \dots, s_{j-1}^{e_{j-1}}, 
u_{\a_n}(\g), s_{j+1}^{e_{j+1}}, \dots, s_m^{e_m} ] \in \BigBS
\]
has $T$-weight in $\Phi^- \setminus \Phi^-_{M \cup
  A})$. (Recall that $e_j = 0$ if $t_j = s_n$ and that $w(-\a_n) \in
\Phi^- \setminus \Phi^-_{M \cup A}$ for all $w \in W_{M \cup A}$.) In
particular for any $\g \in \CM$, $\lim_{z   \to
  \infty} \z^\vee(z) \cdot c(\g) = [\un{e}]$. From the definition of
$S_{w_A}$ as a repelling set we deduce that $f
\circ c$ is contained in $S_{w_A}$ and hence the image of $c$ is contained in
$f^{-1}(S_{w_A})$. Taking derivatives of all such curves we deduce an inclusion
\begin{equation} \label{eq:tangent}
\{ (\la_i)_{i = 1}^\ell \; | \; \la_i = 0 \text{ if } t_i \in
  W_A\} \subset   T_{[\un{e}]}S_{w_A}.
\end{equation}
However both sides have dimension $2a$: the left hand side by
inspection, and the right hand side by \eqref{eq:Sdim}. We deduce that
\eqref{eq:tangent} is an equality, which is \eqref{eq:t2}.

The lemma now follows easily from \eqref{eq:t1}, \eqref{eq:t2} and
\eqref{eq:Twts}.
\end{proof}

\section{Euler class lemma}

\label{sec:euler-class-lemma}

The goal of this section is to prove a lemma which computes the proper
direct image of certain ``combinatorial'' cohomology classes in the
equivariant cohomology of Bott-Samelson resolutions.

We keep the notation of the previous sections. Recall that $T \subset
G = GL_n(\CM)$ denotes the maximal torus of diagonal matrices and $\e_i$ for $1
\le i \le n$ denote the coordinate characters. Given a $T$-space $X$
we denote by $H^*_T(X)$ its equivariant cohomology. (In this section we always take
cohomology with coefficients in $\ZM$.) The Borel
isomorphism gives a canonical isomorphism
\[
H^*_T(pt) = \ZM[\e_1,
\dots, \e_n] = R
\]
with $\deg \e_i = 2$ for $1 \le i \le m$.

Let us fix an expression $\un{w} = t_1t_2 \dots t_m$ and let $BS(\un{w})$ denote the
corresponding Bott-Samelson variety. By the localization theorem the restriction map
\[
H_T^*(BS(\un{w})) \to
H_T^*(BS(\un{w})^T) = \bigoplus_{\un{e} \subset \un{w}} H_T^*([\un{e}]) \]
is injective. In particular, any cohomology class $c \in
H_T^*(BS(\un{w})$ is determined by a tuple $(c_{\un{e}})$ of elements
of $R$ indexed by all subexpressions $\un{e}$ of $\un{w}$.

We say that a class $c \in H_T^*(BS(\un{w})$ is \emph{combinatorial} if there exists
polynomials $f_1, f_2, \dots, f_m$ such that, for any subexpression
$\un{e} = e_1 \dots e_m$ of $\un{w}$, we have
\[
c_{\un{e}} = s_1^{e_1}(f_1s_2^{e_2}(f_2 \dots s_m^{e_m}(f_m) \dots )).
\]
Given a combinatorial $c$ we say that it is \emph{described} by
the polynomials $f_1, f_2, \dots, f_m$.

\begin{ex} \label{ex:ec}
We give an example of a naturally occurring combinatorial cohomology class.
Fix a representation $V$ of $B^m$ and consider the induced bundle
\[
L_V := (P_{t_1} \times \dots \times P_{t_m}) \times_{B^m} V
\]
which is naturally a vector bundle on $BS(\un{w})$ with fibre $V$. For any $1 \le i
\le m$ let $V_i$ denote the restriction of $V$ to the $i^{th}$ copy
$B \subset B^m$ and let $f_i := \det V_i \in \ZM[\e_1, \dots, \e_m]$
denote the product of the characters of $T \subset B$ occurring in
$V_i$.  Then the equivariant Euler class of $L_V$ is combinatorial,
being described by the polynomials $f_1, f_2, \dots, f_m$.
\end{ex}

Let $\un{v} := t_1t_2 \dots t_{m-1}$ be the expression obtained by
ignoring the last term of $\un{w}$. The projection map 
$P_{t_1} \times \dots \times P_{t_{m-1}} \times P_{t_m} \to P_{t_1}
\times \dots \times P_{t_{m-1}}$ induces a morphism
\[
r : BS(\un{w}) \to BS(\un{v})
\]
which is easily seen to be a $\PM^1$-fibration. Given a subexpression
$\un{e}$ of $\un{v}$ we obtain two subexpressions of $\un{w}$ by
appending either a $0$ or a $1$ to $\un{e}$. We denote these
subexpressions simply by $\un{e}0$ and $\un{e}1$. The $T$-fixed points
in the fibre $r^{-1}([\un{v}])$ are precisely the points $\un{e}0$ and $\un{e}1$.

\begin{prop} \label{prop:push}
Suppose that $c \in H_T^*(BS(\un{w})$ is a combinatorial
  class described by $f_1,f_2, \dots, f_m$. Then $r_!(c)$ is
  also combinatorial and is described by $g_1, \dots, g_{m-1}$ where
\[
g_i := \begin{cases} f_i & \text{if $i < m-1$,} \\
f_{m-1}\partial_{t_m}(f_m) & \text{if $i = m-1$.} \end{cases}
\]
\end{prop}

The following well-known lemma provides the key calculation:

\begin{lem}\label{lem:P1}
  Suppose that $X = \PM^1$ with non-trivial linear $T$-action and
  weights at 0 and $\infty$ given by $-\g$ and $\g$ respectively. For any
  class $g = (g_0, g_\infty) \in H_T^*(\PM^1)$ we have
\[
p_!(g) = \frac{g_0 - g_\infty}{\g}
\]
where $p : X \to \pt$ is the projection.
\end{lem}

\begin{proof}
 By the localization theorem $H_T^*(\PM^1)$ identifies with pairs
 $(g_0, g_\infty) \in R \oplus R$ such that $g_0 - g_\infty \in
 (\g)$. It is easy to see that it is free over $R$ with generators in
 degree $0, 2$. Hence $p_!$ is determined by what what it does to the $R$-basis
 $(1,1)$ and $(0, \g)$. However it must annihilate $(1,1)$ for degree
 reasons and must send $(-\g, \g)$ to 2 (the Euler characteristic of
 $\PM^1$). Hence $p_!$ must be given by the above formula.
\end{proof}

\begin{proof}[Proof of Proposition \ref{prop:push}.]
  We claim that the localization of the push-forward of $c$
  at the point $[\un{e}]$ (for $\un{e}$ a subexpression of $\un{v}$) is given by
  \begin{equation} \label{eq:step1}
  (r_!c)_{\un{e}} = \frac{ c_{\un{e} 0} - c_{\un{e} 1} }{t_1^{e_1}
    \dots t_{m-1}^{e_{m-1}} (\a_{t_m})}.
  \end{equation}
As remarked above, $f^{-1}([\un{v}])$ is isomorphic to 
$\PM^1$. Moreover, by \eqref{eq:Twts} the
$T$-weights at the $T$-fixed points $\un{e}0$ and $\un{e}1$ are $t_1^{e_1}
\dots t_{m-1}^{e_{m-1}}(-\alpha_{t_m})$ and $t_1
\dots t_{m-1}^{e_{m-1}}(\alpha_{t_m})$ respectively. The 
equality in \eqref{eq:step1} now
follows by Lemma \ref{lem:P1} and proper base change.

By our assumption that $c$ is combinatorial and described by
$f_1, \dots, f_m$ we can rewrite the right hand side of
\eqref{eq:step1} as
\[
t_1^{e_1}(f_1
  \dots t_{m-1}^{e_{m-1}} (f_{m-1}\left ( \frac{f_m -
      t_mf_m}{\alpha_{t_m}} \right )) \dots ) = 
t_1^{e_1}(f_1
  \dots t_{m-1}^{e_{m-1}}( f_{m-1}\partial_{t_m}(f_m)) \dots )
\]
which is what we wanted to show.
\end{proof}

By iterating the above proposition to the maps $BS(\un{w}) \to BS(\un{v}) \to \dots \to \pt$
we deduce:

\begin{cor} \label{cor:ecl}
  Let $c \in H_T^*(BS(\un{w})$ be a combinatorial class described by
  \[f_1, f_2, \dots, f_m.\]
Then
\[
p_!(c) = \partial_{t_1}(f_1\partial_{t_2}(f_2
\dots \partial_{t_m}(f_m) \dots ))
\]
where $p : BS(\un{w}) \to \pt$ denotes the projection.
\end{cor}

\begin{remark}
  Corollary \ref{cor:ecl} seems to
  be very useful for calculating the proper push-forward of Euler
  classes of vector bundles on Bott-Samelson varieties (see Example
  \ref{ex:ec}). This explains the title of this section. In the next
  section we will see another example of its utility.
\end{remark}

\section{Proof of the main theorem}

Finally, we return to the setting of \S\ref{sec:comb} and
\S\ref{sec:geometry}. Recall our reduced expression $\un{x}$ for $s
\in S_N$, our Schubert variety $\overline{X}_x \subset G/P_M$, our
resolution of singularities
\[
f : \BigBS \to \overline{X}_x
\]
and the normal slice $S_{w_A} \subset G/P_M$ to the Schubert cell
$X_{w_A} \subset G/P_M$. We saw in \S\ref{sec:geometry} that we are in
the miracle situation. Namely, that the fibre $F := f^{-1}(w_A)$ is
smooth and irreducible, and is half-dimensional inside $f^{-1}(S_{w_A})$. 

Moreover we saw in Corollary \ref{cor:bsiso} that we have a
$T$-equivariant isomorphism
\[
\phi : BS(\un{w}_m \dots \un{w}_1)\simto F.
\]
Let us define polynomials $f_i \in R$ for $i = 1, \dots, \sum
\ell(w_i) = a $ by
\begin{align*}
  f_1 = f_2 = \dots = f_{\ell(w_m)-1} = 1, &\quad f_{\ell(w_m)} =
  (\e_n-\e_{n+1}) \dots (\e_n - \e_{n+a_m}), \\
   f_{\ell(w_m) + 1} = \dots = f_{\ell(w_m) + \ell(w_{m-1})-1} = 1, &
   \quad f_{\ell(w_m)+ \ell(w_{m-1})} =
   (\e_n-\e_{n+1}) \dots (\e_n - \e_{n+a_{m-1}}), \\
\vdots & \quad \vdots \\
   f_{\ell(w_m) + \dots \ell(w_2) + 1} = \dots = f_{a-1} = 1, & \quad f_a =
   (\e_n-\e_{n+1}) \dots (\e_n - \e_{n+a_1}). \\
\end{align*}

\begin{lem}  \label{lem:polys}
Under the isomorphism $\phi$  above the equivariant Euler class of the
normal bundle of $F \subset f^{-1}(S_{w_A})$ is combinatorial, and is
described by the polynomials $f_1, \dots, f_a$.
\end{lem}

\begin{proof} Recall that the localization in $T$-equivariant cohomology
  of an Euler class of an equivariant vector bundle is given by
  the product of the $T$-weights at each fixed point. Hence, by
Lemma \ref{lem:normal} the localization of the Euler class of the
normal bundle at a $T$-fixed point $\un{e}$ in $F$ is given by the product
\begin{equation}
  \label{eq:n1}
  E_{\un{e}} := \prod_{1 \le j \le \ell \atop t_j = s_n} t_1^{e_1} \dots
  t_j^{e_j}(-\alpha_n).
\end{equation}
To complete the proof, we will argue that we can rewrite the above
expression to yield a combinatorial class described by the above
polynomials.

Recall that $\un{x}$ has the form
\[
\un{x} = \un{w}_{m} \un{z}_m \dots \un{w}_2 \un{z}_2 \un{w}_1 \un{z}_1
\]
where each $\un{w}_i$ is an expression in $M$, and each $\un{z}_i$ an
expression in $A \cup \{ s_n \}$. Let us write
\[
\un{e} =
\un{e}_m'\un{f}_m'\dots\un{e}_{2}'\un{f}_{2}' 
\un{e}_1'\un{f}_1', \]
where each $\un{e}_i'$ (resp. $\un{f}_i'$) is the corresponding
subexpression of $\un{w}_i$ (resp. $\un{z}_i$). By Lemma \ref{lem:sub}
the fixed points $[\un{e}]$ in $\BigBS$ correspond to those
subexpressions $\un{e}$ of $\un{x}$ with $e_i = 0$ if $t_i = s_n$ and
$e_1 = 1$ if $t_i \in A$. An alternative way of saying this is that
we have no choice for the subexpressions $\un{f}_i'$: if we bracket $\un{z}_i$ as
\[
  (s_ns_{n+1} \dots s_{n+a_1 + \dots + a_i- 1}) \dots (s_ns_{n+1} \dots
  s_{n+a_1+ \dots + a_{i-1} + 1}) (s_ns_{n+1} \dots s_{n+a_1 + \dots +a_{i-1}}),\\
\]
then $\un{f}_i'$ has the form
\[
(01 \dots 1) \dots (01 \dots 1)(01 \dots 1).
\]
Hence we can rewrite $E_{\un{e}}$ as
\[
\un{w}_m^{\un{e}_m'} (n_m) \cdot \un{w}_m^{\un{e}_m'}
\un{w}_{m-1}^{\un{e}_{m-1}'}  (n_{m-1}) \cdot \quad \dots \quad\cdot 
\un{w}_m^{\un{e}_m'}
\un{w}_{m-1}^{\un{e}_{m-1}'}  \dots 
\un{w}_{1}^{\un{e}_{1}'}   (n_{1}) 
\]
where
\[
n_i = (\e_n - \e_{n+1})(\e_n - \e_{n+2}) \dots (\e_n - \e_{n+a_i}).
\]
For example:
\begin{align*}
  n_1 & = (-\alpha_n)(s_{n}s_{n+1} \dots s_{n+a_1-1}(-\alpha_n)) \dots
  ((s_{n+1} \dots s_{n+a_1-1}) \dots (s_ns_{n+1})(-\alpha_n)) \\
& = (\e_n - \e_{n+1})(\e_n - \e_{n+2}) \dots (\e_n - \e_{n+a_1})
\end{align*}
The lemma now follows.
\end{proof}

\begin{remark}
  One might hope that the reason that the Euler class of
  the normal bundle to $F \subset f^{-1}(S_{w_A})$ is combinatorial is
  because it is an induced bundle as in Example \ref{ex:ec}. I was
  unable to decide whether this is the case.
\end{remark}

\begin{proof}[Proof of Theorem \ref{thm:main}]
We can now complete the proof of theorem \ref{thm:main}. If we denote
by $c \in H_T^*(F)$ the Euler class of the normal bundle to $F \subset
f^{-1}(w_AP_M/P_M)$ and $p : F \to \pt$ denotes the projection then, by Lemma \ref{lem:polys} and Corollary
\ref{cor:ecl} we have
\begin{align*}
  p_!(c) &= \partial_{w_m}(n_m \partial_{w_{m-1}}(n_{m-1} \dots
    (\partial_{w_1}n_1) \dots )).
\end{align*}
By repeated application of Lemma \ref{lem:inv} below
\begin{equation*}
\partial_{w_m}(n_m \partial_{w_{m-1}}(n_{m-1} \dots
    (\partial_{w_1}n_1) \dots )) = \partial_{w_m}(\e_n^{a_m} \partial_{w_{m-1}}(\e_n^{a_{m-1}} \dots
    (\partial_{w_1}\e_n^{a_1}) \dots )) = C
\end{equation*}
where $C$ is as in the introduction. This completes the proof.  
\end{proof}

\begin{lem} \label{lem:inv}
  Consider an expression of the form
\[
D = \partial_{u_1}(g_1\partial_{u_2}(g_2 \dots \partial_{u_m}(g_m +
h_m^+g_m') \dots )) \in R
\]
with $u_i \in W_M$, $g_i \in R$, $g_m' \in R$ and $h^+_m \in
R^{W_M}$. If $D \in \ZM$ and $h^+_m$ is of degree $> 0$ then
\[
D = \partial_{u_1}(g_1\partial_{u_2}(g_2 \dots \partial_{u_m}(g_m) \dots )).
\]
\end{lem}

\begin{proof}
  Because $\partial_{u_i}(h_m^+g) = h_m^+\partial_{u_i}(g)$ for all $g
  \in R$ we have
\begin{gather*}
D = \partial_{u_1}(g_1\partial_{u_2}(g_2 \dots \partial_{u_m}(g_m +
h_m^+g_m') \dots )) = \\
= \partial_{u_1}(g_1\partial_{u_2}(g_2
\dots \partial_{u_m}(g_m) \dots )) + h_m^+\partial_{u_1}(g_1\partial_{u_2}(g_2 \dots \partial_{u_m}(g_m') \dots )).
\end{gather*}
Because $D \in \ZM$, the term  $\partial_{u_1}(g_1\partial_{u_2}(g_2
\dots \partial_{u_m}(g_m') \dots ))$ is of negative degree, and hence
is zero. The lemma follows.
\end{proof}

 \def\cprime{$'$} \def\cprime{$'$} \def\cprime{$'$} \def\cprime{$'$}

\end{document}